\documentclass{amsart}

\usepackage{color}
 \newtheorem{thm}{Theorem}
 
 \newtheorem{lemma}{Lemma}
 \newtheorem{prop}{Proposition}
 \theoremstyle{definition}
 \newtheorem{defn}{Definition}
 \theoremstyle{remark}
 \newtheorem{rem}{Remark}

\begin{document}

\title[Carath\'{e}odory theorems for Slice Regular Functions]
{Carath\'{e}odory theorems for Slice Regular Functions}

\author[G. B. Ren]{Guangbin Ren}
\address{Guangbin Ren, Department of Mathematics, University of Science and
Technology of China, Hefei 230026, China}

\email{rengb$\symbol{64}$ustc.edu.cn}

\author[X. P. Wang]{Xieping Wang}
\address{Xieping Wang, Department of Mathematics, University of Science and
Technology of China, Hefei 230026,China}

\email{pwx$\symbol{64}$mail.ustc.edu.cn}

\thanks{This work was supported by the NNSF  of China (11071230), RFDP (20123402110068).}

\keywords{Quaternion, Slice Regular Functions, Carath\'{e}odory Theorems.}
\subjclass[2010]{30G35. 32A26}

\begin{abstract}
In this paper a quaternionic sharp version of the Carath\'{e}odory  theorem is established for slice regular functions with positive real part, which strengthes a weaken version recently established  by D. Alpay et. al.  using the  Herglotz integral formula. Moreover, the restriction of positive real part can be relaxed so that the theorem becomes  the quaternionic version of the Borel-Carath\'{e}odory  theorem. It turns out that  the two theorems are equivalent.
\end{abstract}

\maketitle

\section{Introduction}
The celebrated  Carath\'{e}odory theorem for holomorphic functions with positive real part  plays an important role in the geometric function theory of one complex variable (see \cite{C,Du,GG}):

\begin{thm}\label{CT1}
Let $f:\mathbb D\rightarrow \mathbb C$ be a holomorphic function such that $f(z)=1+\sum\limits_{n=1}^\infty a_nz^n$ and ${\rm{Re}}f(z)>0$, then for all $z\in\mathbb D$,
\begin{eqnarray}\label{eq:1}
\frac{1-|z|}{1+|z|}\leq {\rm{Re}}f(z)\leq |f(z)|\leq \frac{1+|z|}{1-|z|};
\end{eqnarray}
 and
 \begin{eqnarray}\label{eq:11an}|a_n|\leq 2,\qquad \forall \ n\in\mathbb N.
 \end{eqnarray}
Moreover, equality holds for the first or the third inequality in $(\ref{eq:1})$ at some point $z_0\neq 0$ if and only if $f$ is of the form $$f(z)=\frac{1+e^{i\theta} z}{1-e^{i\theta} z}\ ,\qquad \forall\,\, z\in \mathbb C,$$ for some $ \theta \in \mathbb R$.

\end{thm}

A well-known Borel-Carath\'{e}odory theorem is a variant of the Carath\'{e}odory theorem relaxing  the restriction of positive real part (For  a Clifford-analytic version, see \cite{GMC}):

\begin{thm}\label{CT2}
Let $f(z)=\sum\limits_{n=0}^{\infty}a_nz^n$ be a holomorphic function on $\mathbb D$, continuous up to the boundary $\partial \mathbb D$. Set $A=\max\limits_{|z|=1}{\rm{Re}}f(z)$, then
\begin{equation}\label{111}
  |a_n|\leq 2(A-{\rm{Re}}f(0)),\qquad\forall\,\,n\in\mathbb N;
\end{equation}
\begin{equation}\label{112}
 |f(z)-f(0)|\leq \frac{2r}{1-r}(A-{\rm{Re}}f(0)),\qquad \forall\,\,|z|\leq r<1.
 \end{equation}
\end{thm}

 It is quite natural to extend the results to higher dimensions or to other function classes.
However, a great challenge may  arise from  extensions to  non-holomorphic function classes, because of the  failure  of closeness  under multiplication and composition.

The purpose of this article is to generalize the above two theorems to the setting of quaternions for slice regular functions. The theory of slice regular functions is initiated recently by Gentili and Struppa \cite{GS1,GS2}. The detailed up-to-date theory appears in the monographs \cite{GSS, Co2}.
 In particular, the slice regular product was introduced in the setting of quaternions in \cite{GSS1} for slice regular power series and in \cite{CGSS} for slice regular functions on symmetric slice domains.
 The geometric theory of slice regular functions of one quaternionic variable
has been studied in \cite{GS,CGS1,CGS2}.
 Recently,  the authors \cite{RW} establish the growth  and distortion theorems for slice regular extensions of normalized univalent holomorphic functions with the help of a so-called convex combination identity.

Our main results in this article  are the  quaternionic versions of  the Carath\'{e}odory  theorem as well as  the  Borel-Carath\'{e}odory theorem  for slice regular functions:

\begin{thm}\label{CT3}
Let $\mathbb B$ the open  unit ball in the quaternions $\mathbb H$.
If $f:\mathbb B\rightarrow \mathbb H$ is a regular function for which  $f(q)=1+\sum\limits_{n=1}^\infty q^n a_n$ and ${\rm{Re}}f(q)>0$, then
\begin{eqnarray}\label{eq:11}
\frac{1-|q|}{1+|q|}\leq  {\rm{Re}}f(q)\leq |f(q)|\leq \frac{1+|q|}{1-|q|}, \qquad\forall\ q\in\mathbb B,
\end{eqnarray}
 and \begin{eqnarray}\label{eq:11an}|a_n|\leq 2,\qquad \forall \ n\in\mathbb N.  \end{eqnarray}
Moreover, equality holds for the first or the third inequality in $(\ref{eq:11})$ at some point $q_0\neq 0$ if and only if $f$ is of the form $$f(q)=(1-q e^{I\theta})^{-\ast}\ast(1+qe^{I\theta}),\qquad \forall\,\, q\in \mathbb B,$$ for some  $\theta \in \mathbb R$ and some $I\in \mathbb H$
 with $I^2=-1$.

Equality holds in $(\ref{eq:11an})$  for some  $n_0\in \mathbb N$, i.e. $|a_{n_0}|=2$ if and only if
 $$a_{kn_0}=2\Big(
\dfrac{a_{n_0}}{2}\Big)^{k}, \qquad \forall\  k\in\mathbb N.
$$
In particular,
 $|a_1|=2$ if and only if
$$f(q)=(1-q e^{I\theta_0})^{-\ast}\ast(1+qe^{I\theta_0}),\qquad \forall\,\, q\in \mathbb B,$$ for some  $\theta_0 \in \mathbb R$ and some $I\in \mathbb H$
 with $I^2=-1$.
\end{thm}

\begin{thm}\label{CT4}
Let $f(q)=\sum\limits_{n=0}^{\infty}q^n a_n$ be a slice regular function on $\mathbb B$. If $A:=\sup\limits_{q\in \mathbb B}{\rm{Re}}f(q)<+\infty$, then
\begin{equation}\label{114}
  |a_n|\leq 2\big(A-{\rm{Re}}f(0)\big),\qquad\forall\,\,n\in\mathbb N;
\end{equation}

\begin{equation}\label{115}
 |f(q)-f(0)|\leq \frac{2r}{1-r}\big(A-{\rm{Re}}f(0)\big),\qquad \forall\,\,|q|\leq r<1;
 \end{equation}
 \begin{equation}\label{155}
{\rm{Re}}f(q)\leq \frac{2r}{1+r}A+\frac{1-r}{1+r}{\rm{Re}}f(0),\qquad \forall\,\,|q|\leq r<1;
\end{equation}
 \begin{equation}\label{116}
  |f^{(n)}(q)|\leq\frac{2n!}{(1-r)^{n+1}}\big(A-{\rm{Re}}f(0)\big),\qquad \forall\,\,|q|\leq r<1,\,\,n\in\mathbb N.
 \end{equation}
\end{thm}

We remark as pointed to us by Sabadini  that the weak inequalities  $$|{\rm{Re}}(a_n)|\leq 2$$ other than (\ref{eq:11an})  can only be deduced  with the approach as in \cite{Alpay} via the  Herglotz integral formula, since $d\mu_2(t)$ in Corollary 8.4 of   \cite{Alpay}  may  not be  a non-negative measure  in general.
Incidentally, the fact that $|{\rm{Re}}(a_n)|\leq 2$ can also be proved by using the splitting lemma for slice regular functions and the Schwarz integral formula for holomorphic functions of one complex variable. However, the two methods can not be used to prove the sharp version. In this article we can apply  the approach of finite average to deduce the  strong  version as in (\ref{eq:11an}). In addition, a weaken inequality than (\ref{115}) in Theorem \ref{CT4} has been proved in \cite{CGS1} and our new approach allows to strengthen the statements. Finally, we point out   that Theorems \ref{CT3} and  \ref{CT4} turn out to be equivalent.

\section{Preliminaries}

We recall in this section some preliminary definitions and results on slice regular functions.
To have a more complete insight on the theory, we refer the reader to \cite{GSS}.

Let $\mathbb H$ denote the non-commutative, associative, real algebra of quaternions with standard basis $\{1,\,i,\,j, \,k\}$,  subject to the multiplication rules
$$i^2=j^2=k^2=ijk=-1.$$
 Every element $q=x_0+x_1i+x_2j+x_3k$ in $\mathbb H$ is composed by the \textit{real} part ${\rm{Re}}\, (q)=x_0$ and the \textit{imaginary} part ${\rm{Im}}\, (q)=x_1i+x_2j+x_3k$. The \textit{conjugate} of $q\in \mathbb H$ is then $\bar{q}={\rm{Re}}\, (q)-{\rm{Im}}\, (q)$ and its \textit{modulus} is defined by $|q|^2=q\overline{q}=|{\rm{Re}}\, (q)|^2+|{\rm{Im}}\, (q)|^2$. We can therefore calculate the multiplicative inverse of each $q\neq0$ as $ q^{-1} =|q|^{-2}\overline{q}$.
 Every $q \in \mathbb H $ can be expressed as $q = x + yI$, where $x, y \in \mathbb R$ and
$$I=\dfrac{{\rm{Im}}\, (q)}{|{\rm{Im}}\, (q)|}$$
 if ${\rm{Im}}\, q\neq 0$, otherwise we take $I$ arbitrarily such that $I^2=-1$.
Then $I $ is an element of the unit 2-sphere of purely imaginary quaternions,
$$\mathbb S=\big\{q \in \mathbb H:q^2 =-1\big\}.$$

For every $I \in \mathbb S $ we will denote by $\mathbb C_I$ the plane $ \mathbb R \oplus I\mathbb R $, isomorphic to $ \mathbb C$, and, if $\Omega \subseteq \mathbb H$, by $\Omega_I$ the intersection $ \Omega \cap \mathbb C_I $. Also, for $R>0$, we will denote the open ball centred at the origin with radius $R$ by
$$B(0,R)=\big\{q \in \mathbb H:|q|<R\big\}.$$

We can now recall the definition of slice regularity.
\begin{defn} \label{de: regular} Let $\Omega$ be a domain in $\mathbb H$. A function $f :\Omega \rightarrow \mathbb H$ is called \emph{slice} \emph{regular} if, for all $ I \in \mathbb S$, its restriction $f_I$ to $\Omega_I$ is \emph{holomorphic}, i.e., it has continuous partial derivatives and satisfies
$$\bar{\partial}_I f(x+yI):=\frac{1}{2}\left(\frac{\partial}{\partial x}+I\frac{\partial}{\partial y}\right)f_I (x+yI)=0$$
for all $x+yI\in \Omega_I $.
 \end{defn}
As shown in \cite {CGSS}, a class of domains, the so-called symmetric slice domains naturally qualify as  domains of definition  of slice regular functions.

\begin{defn} \label{de: domain}
Let $\Omega$ be a domain in $\mathbb H $.  $\Omega$ is called a \textit{slice domain}  if $\Omega$ intersects the real axis and $\Omega_I$  is a domain
of $ \mathbb C_I $  for any $I \in \mathbb S $.

Moreover,  if  $x + yI \in \Omega$ implies $x + y\mathbb S \subseteq \Omega $ for any $x,y \in \mathbb R $ and $I\in \mathbb S$, then
 $\Omega$  is called a \textit{symmetric slice domain}.
\end{defn}

From now on, we will omit the term `slice' when referring to slice regular functions and will focus mainly on regular functions on $B(0,R)=\big\{q \in \mathbb H:|q|<R\big\}$.
For regular functions the natural definition of derivative is given by the following (see \cite{GS1,GS2}).
\begin{defn} \label{de: derivative}
Let $f :B(0,R) \rightarrow \mathbb H$  be a regular function. The \emph{slice derivative} of $f$ at $q=x+yI$
is defined by
$$\partial_I f(x+yI):=\frac{1}{2}\left(\frac{\partial}{\partial x}-I\frac{\partial}{\partial y}\right)f_I (x+yI).$$
 \end{defn}
Notice that the operators $\partial_I$ and $\bar{\partial}_I $ commute, and $\partial_I f=\frac{\partial f}{\partial x}$ for regular functions. Therefore, the slice derivative of a regular function is still regular so that we can iterate the differentiation to obtain the $n$-th
slice derivative
$$\partial^{n}_I f=\frac{\partial^{n} f}{\partial x^{n}},\quad\,\forall \,\, n\in \mathbb N. $$

In what follows, for the sake of simplicity, we will direct denote the $n$-th slice derivative $\partial^{n}_I f$ by $f^{(n)}$ for every $n\in \mathbb N$.

As shown in \cite{GS2}, a quaternionic power series $\sum\limits_{n=0}^{\infty}q^n a_n$ with $\{a_n\}_{n \in \mathbb N} \subset \mathbb H$ defines a regular function in its domain of convergence, which proves to be a open ball $B(0,R)$ with $R$ equal to the radius of convergence of the power series. The converse result is also true.
\begin{thm}{\bf(Taylor Expansion)}\label{eq:Taylor}
A function f is regular on $B = B(0,R) $ if and only if $f$ has a power series expansion
$$f(q)=\sum\limits_{n=0}^{\infty}q^n a_n\quad with \quad a_n=\frac{f^{(n)}(0)}{n!}.$$
\end{thm}

A fundamental result in the theory of regular functions is described by the splitting lemma (see \cite{GS2}), which relates slice regularity to classical holomorphy.
\begin{lemma}{\bf(Splitting Lemma)}\label{eq:Splitting}
Let $f$ be a regular function on $B = B(0,R)$, then for any $I\in \mathbb S$ and any $J\in \mathbb S$ with $J\perp I$, there exist two holomorphic functions $F,G:B_I\rightarrow \mathbb C_I$ such that for every $z=x+yI\in B_I $, the following equality holds
$$f_I(z)=F(z)+G(z)J.$$
\end{lemma}

The following version of the identity principle is one of the first consequences
(as shown in \cite{GS2}).
\begin{thm}{\bf(Identity Principle)}\label{th:IP-theorem}
Let $f$ be a regular function on $B = B(0,R)$. Denote by $\mathcal{Z}_f$ the zero set of $f$, $$\mathcal{Z}_f=\big\{q\in B:f(q)=0\big\}.$$
If there exists an $I\in \mathbb S$ such that $B_I\cap \mathcal{Z}_f$ has an accumulation point in $B_I$, then $f$ vanishes identically on $B$.
\end{thm}

\begin{rem}\label{R-extension}
Let $f_I$ be a holomorphic function on a disc $B_I=B(0,R)\cap \mathbb C_I$ and let its
power series expansion take the form
$$f_I(z)=\sum\limits_{n=0}^{\infty}z^na_n$$
with $\{a_n\}_{n\in \mathbb N}\subset\mathbb H$. Then the unique regular extension of $f_I$ to the whole ball $B(0,R)$ is the function
defined by
$$f(q):={\rm{ext}}(f_I)(q)=\sum\limits_{n=0}^{\infty}q^na_n.$$
The uniqueness is guaranteed by the identity principle \ref{th:IP-theorem}.
\end{rem}

In general, the pointwise product of two regular functions is not a regular function. To guarantee the regularity of the product we need to introduce a new multiplication operation, the regular product (or $\ast$-product). On open balls centred at the origin, the $\ast$-product of two regular functions is defined by means of their power series
expansions (see \cite{GSS1}).

\begin{defn}\label{R-product}
Let $f$, $g:B=B(0,R)\rightarrow \mathbb H$ be two regular functions and let
$$f(q)=\sum\limits_{n=0}^{\infty}q^na_n,\qquad g(q)=\sum\limits_{n=0}^{\infty}q^nb_n$$
be their series expansions. The regular product (or $\ast$-product) of $f$ and $g$ is the function defined by
$$f\ast g(q)=\sum\limits_{n=0}^{\infty}q^n\bigg(\sum\limits_{k=0}^n a_kb_{n-k}\bigg)$$
regular on $B$.
\end{defn}
Notice that the $\ast$-product is associative and is not, in general, commutative. Its connection with the usual pointwise product is clarified by the following result (see \cite{GSS1}).
\begin{prop} \label{prop:RP}
Let $f$ and $g$ be two regular functions on $B=B(0,R)$. Then for all $q\in B$,
$$f\ast g(q)=
\left\{
\begin{array}{lll}
f(q)g(f(q)^{-1}qf(q)) \qquad \,\,if \qquad f(q)\neq 0;
\\
\qquad  \qquad  0\qquad  \qquad \qquad if \qquad f(q)=0.
\end{array}
\right.
$$
\end{prop}
We remark that if $q=x+yI$ and $f(q)\neq0$, then $f(q)^{-1}qf(q)$ has the same modulus and same real part as $q$. Therefore $f(q)^{-1}qf(q)$ lies in the same 2-sphere $x+y\mathbb S$ as $q$. We obtain then that a zero $x_0+y_0I$ of the function $g$ is not necessarily a zero of $f\ast g$, but an element on the same sphere $x_0+y_0\mathbb S$ does. In particular, a
real zero of $g$ is still a zero of $f\ast g$. To present a characterization of the structure of the zero set of a regular function $f$, we need to introduce the following functions.
\begin{defn} \label{de: R-conjugate}
Let $ f(q)=\sum\limits_{n=0}^{\infty}q^na_n $ be a regular function on $B=B(0,R)$. We define the \emph{regular conjugate} of $f$ as
$$f^c(q)=\sum\limits_{n=0}^{\infty}q^n\bar{a}_n,$$
and the \emph{symmetrization} of $f$ as
$$f^s(q)=f\ast f^c(q)=f^c\ast f(q)=\sum\limits_{n=0}^{\infty}q^n\bigg(\sum\limits_{k=0}^n a_k\bar{a}_{n-k}\bigg).$$
Both $f^c$ and $f^s$ are regular functions on $B$.
\end{defn}
We are now able to define the inverse element of a regular function $f$ with respect to the $\ast$-product. Let $\mathcal{Z}_{f^s}$ denote the zero set of the symmetrization $f^s$ of $f$.
\begin{defn} \label{de: R-Inverse}
Let $f$ be a regular function on $B=B(0,R)$. If $f$ does not vanish identically, its \emph{regular reciprocal} is the function defined by
$$f^{-\ast}(q):=f^s(q)^{-1}f^c(q)$$
regular on $B \setminus \mathcal{Z}_{f^s}$.

\end{defn}
The following result shows that the regular product is nicely related to the pointwise quotient (see \cite{Stop1} and for general case see \cite{Stop2}).
\begin{prop} \label{prop:Quotient Relation}
Let $f$ and $g$ be regular functions on  $B=B(0,R)$. Then for all $q\in B \setminus \mathcal{Z}_{f^s}$, $$f^{-\ast}\ast g(q)=f(T_f(q))^{-1}g(T_f(q)),$$
where $T_f:B \setminus \mathcal{Z}_{f^s}\rightarrow B \setminus \mathcal{Z}_{f^s}$ is defined by
$T_f(q)=f^c(q)^{-1}qf^c(q)$. Furthermore, $T_f$ and $T_f^c$ are mutual inverses so that $T_f$ is a diffeomorphism.
\end{prop}

The following Schwarz lemma was proved in \cite {GS2}.
\begin{lemma}{\bf(Schwarz Lemma)}\label{th:SL-theorem}
Let $f:\mathbb B\rightarrow\mathbb B$ be a regular function. If $f(0)=0$, then
$$|f(q)|\leq|q|$$
for all $q\in \mathbb B$ and $$|f'(0)|\leq1.$$
Both inequalities are strict {\rm{(}}except at $q=0${\rm{)}} unless $f(q)=qu$ for some $u\in \partial\mathbb B$.
\end{lemma}
The following Leibniz  rule was proved in \cite{GSS2}.
\begin{prop}{\bf(Leibniz rule)}\label{th:Leibniz-rule}
Let $f$ and $g$ be regular functions on  $B=B(0,R)$. Then
$$(f\ast g)'=f'\ast g+f\ast g'.$$

\end{prop}

\section{Proof of the Main Theorems}
In this section, we give the proofs of Theorems \ref{CT3} and \ref{CT4}.

\begin{proof}[Proof of Theorem $\ref{CT3}$]
Denote
 $$g(q)=(f(q)-1)\ast(f(q)+1)^{-\ast}.$$
Notice that
 $$(f(q)-1)\ast(f(q)+1)^{-\ast}=(f(q)+1)^{-\ast}\ast(f(q)-1),$$
 we have
$$g(q)=(f(q)+1)^{-\ast}\ast(f(q)-1).$$
It is evident that   $g$ is regular on $\mathbb B$ and $g(0)=0$. By Proposition \ref{prop:Quotient Relation}, we have
$$g(q)=\big(f\circ T_{1+f}(q)+1\big)^{-1}\big(f\circ T_{1+f}(q)-1\big). $$
This together with the fact that ${\rm{Re}}(f(q))>0$ yields $g(\mathbb B)\subseteq\mathbb B$. Therefore,
 Lemma \ref{th:SL-theorem} implies  that
\begin{eqnarray}\label{eq:21}
|g(q)|\leq|q|, \qquad \forall\ q\in\mathbb B,
\end{eqnarray}
and
\begin{eqnarray}\label{eq:22}
|g'(0)|\leq1.
\end{eqnarray}
From the very definition of $g$ and Proposition $\ref{prop:Quotient Relation}$, it follows that
\begin{eqnarray}\label{eq:23}
f(q)&=&(1-g(q))^{-\ast}\ast(g(q)+1)\notag
\\
&=&\big(1-g\circ T_{1-g}(q)\big)^{-1}\big(g\circ T_{1-g}(q)+1\big).
\end{eqnarray}
Recall that $$|T_{1-g}(q)|=|q|, \qquad\forall\ q\in\mathbb B,$$ from $(\ref{eq:21})$ and $(\ref{eq:23})$ we obtain
\begin{eqnarray}\label{eq:24}
|f(q)|
\leq\frac{1+|T_{1-g}(q)|}{1-|T_{1-g}(q)|}=\frac{1+|q|}{1-|q|}
\end{eqnarray}
and similarly
\begin{equation}\label{eq:25}
{\rm{Re}}(f(q))=
\frac{1-|g\circ T_{1-g}(q)|^2}{|1-g\circ T_{1-g}(q)|^2}
\geq\frac{1-|q|^2}{(1+|q|)^2}
= \frac{1-|q|}{1+|q|}.
\end{equation}

If equality holds  in $(\ref{eq:24})$ or $(\ref{eq:25})$ at some point $q_0\neq 0$, then  it must be true that $|g(q_0)|=|q_0|$. Lemma \ref{th:SL-theorem} thus implies  that $$g(q)=q e^{I\theta},\quad \forall \,\,q\in \mathbb B,$$
for some $I\in \mathbb S$ and some $\theta\in\mathbb R$, and hence
$$f(q)=(1-q e^{I\theta})^{-\ast}\ast(1+qe^{I\theta}),\quad \forall \,\,q\in \mathbb B.$$
The converse part can be easily verified by a simple calculation.

We now come to prove  the assertion  $$|a_n|\leq2,\quad n=1,2,\ldots,$$
where $a_n$ are the Taylor coefficients of $f$ on the open unit ball $\mathbb B$.

First, from the very definition of $g$ it follows that
$$f\ast(1-g)=1+g,$$
from which as well as Proposition \ref{th:Leibniz-rule} we obtain that
$$f'\ast(1-g)=(f+1)\ast g'.$$
We now  evaluate  the preceding identity at $q=0$. Since $f(0)=1$ and $g(0)=0$, applying Proposition \ref{prop:RP} we obtain
\begin{eqnarray}\label{eq:26}
 f'(0)=2g'(0),
 \end{eqnarray}
 which together with (\ref{eq:22}) yields
\begin{eqnarray}\label{eq:27}
|a_1|=|f'(0)|\leq2.
 \end{eqnarray}

If $|a_1|=2$, i.e., $|g'(0)|=1$, it follows from Lemma \ref{th:SL-theorem} that $$g(q)=q e^{I\theta},\quad \forall \,\,q\in \mathbb B,$$
for some $I\in \mathbb S$ and some $\theta\in\mathbb R$, and hence
\begin{eqnarray}\label{eq:28}
f(q)=(1-q e^{I\theta})^{-\ast}\ast(1+qe^{I\theta}),\quad \forall \,\,q\in \mathbb B.
 \end{eqnarray}

Now we want to prove that $|a_{n_0}|\leq 2$ for any fixed  $n_0\geq 2$.
To this end, we need to
construct a regular function $\varphi$ with the same properties as $f$, whose Taylor coefficient of the first degree term is $a_{n_0}$ so that   we would conclude  that $|a_{n_0}|\leq 2$.

Starting from the Taylor expansion of  $f$, given by  $f(q)=1+\sum\limits_{n=1}^\infty q^n a_n$,
we set
\begin{equation}\label{eq:31}
\varphi(q)=1+\sum\limits_{m=1}^\infty q^{m} a_{mn_0}=1+qa_{n_0}+q^2a_{2n_0}+\cdots
\end{equation}
and
\begin{equation}\label{eq:381}
h(q)=\varphi(q^{n_0})
\end{equation}
One can see from the radii of convergence of Taylor expansions that $\varphi$ and $h$ are slice regular functions on $\mathbb B$.

Now we claim that the restriction $h_I$ of $h$ to  $\mathbb B_I$ with $I\in\mathbb S$ is exactly given by
\begin{equation}\label{eq:25-claim}  h_I(z_I)=\frac 1 {n_0} \Big(f_I(z_I)+f_I(z_I\omega_I)+\cdots +f_I(z_I\omega_I^{n_0-1})\Big),\quad \forall \,\,z_I\in \mathbb B_I,\end{equation}
where $f_I$ is the restriction of $f$ to $\mathbb B_I$
and  $\omega_I \in\mathbb C_I$ is any  quaternionic primitive $n_0$-th root of unity.
 From (\ref{eq:381}) the claim is equivalent to the identity
\begin{equation}\label{eq:259} \frac 1 {n_0} \sum_{k=0}^{n_0-1}f_I(z_I\omega_I^{k})=1+\sum\limits_{m=1}^\infty q^{mn_0} a_{mn_0}.\end{equation}
To prove this, we apply the power series expansion of $f$ and obtain
\begin{equation}\label{eq:29}
\begin{split}\frac 1 {n_0} \sum_{k=0}^{n_0-1}f_I(z_I\omega_I^{k})=
&1+\frac 1 {n_0} \Big(\sum\limits_{m=1}^\infty z_I^m a_m+\sum\limits_{m=1}^\infty z_I^m\omega_I^m a_m+\cdots +\sum\limits_{m=1}^\infty z_I^m\omega_I^{(n_0-1)m} a_m\Big)\\
=&1+\frac 1 {n_0}\sum\limits_{m=1}^\infty z_I^m\bigg(\sum\limits_{k=0}^{n_0-1}\omega_I^{km}\bigg)a_m.
\end{split}
\end{equation}

Since  $\omega_I$ is a  primitive $n$-th root of unity in the plane $\mathbb C_I$, we have
\begin{eqnarray}\label{eq:23e8}
\frac 1 {n_0}\sum\limits_{k=0}^{n_0-1}\omega_I^{km}=
\left\{
 \begin{array}{lll}
1, \qquad n_0\mid m;
\\
0,\qquad \mbox{otherwise}.
 \end{array}
  \right.
   \end{eqnarray}
Indeed,  if $n_0\mid m$, then $\omega_I^m=1$ such that each summand in (\ref{eq:23e8}) equals $1$,  so is the average. Otherwise, then
\begin{eqnarray*}
\sum\limits_{k=0}^{n_0-1}\omega_I^{km}=\frac{1-\omega_I^{n_0m}}{1-\omega_I^{m}}
=0
 \end{eqnarray*}
and (\ref{eq:23e8}) holds true. Inserting (\ref{eq:23e8}) into (\ref{eq:29}), we get (\ref{eq:259}) and this proves the claim.

We want to show that
\begin{equation}\label{eq:2re9} {\rm{Re}}(\varphi(q))>0\end{equation} for any $q\in\mathbb B.$
Let   $q\in\mathbb B$ be given and take  $u\in\mathbb B$ such that $u^{n_0}=q$. Hence due to (\ref{eq:381}) we have
 $$\varphi(q)=h(u).$$
 To prove (\ref{eq:2re9}),
  It suffices to prove that
  $${\rm{Re}}(h(u))>0, \qquad  \forall\ u\in\mathbb B.$$
Thus we only need   to show the result for the  restrictions of $h$ to any complex plane $h_I$, i.e.,
   $${\rm{Re}}(h_I(z_I))>0, \qquad \forall\ z_I\in B_I, \ \forall\ I\in \mathbb S. $$
  This follows  easily from (\ref{eq:25-claim}) and  the assumption that ${\rm{Re}}(f(q))>0$.

 Therefore, the regular function $\varphi$ we have constructed via $f$ has the same properties as $f$, whose Taylor coefficient of the first degree term is $a_{n_0}$. Consequently, we obtain that $|a_{n_0}|\leq 2$ for any $n_0\in\mathbb N$.

Finally, if there exists $n_0\in \mathbb N$ such that $|a_{n_0}|=2$, then the argument similar to the extremal case for $n=1$ implies that
$$\varphi(q)=(1-q e^{I\theta_0})^{-\ast}\ast(1+qe^{I\theta_0})
=1+2\sum\limits_{m=1}^\infty q^m e^{Im\theta_0},
\quad \forall \,\,q\in \mathbb B,$$
for some $I\in \mathbb S$ and some $\theta_I\in\mathbb R$, which together with (\ref{eq:31})  implies that
$$a_n=
\left.
\begin{array}{lll}
2\Big(\dfrac{a_{n_0}}{2}\Big)^{\frac {n}{n_0}} \qquad if \qquad n_0\mid n,
\end{array}
\right.
$$
where $a_{n_0}=2e^{In_0\theta_0}$. In particular, when $n_0=1$, this is clearly equivalent to (\ref{eq:28}). Now the proof is completed.

\end{proof}
To prove Theorem \ref{CT4}, we need a useful  lemma, which have been proved in \cite[Theorem 2.7]{GS3}. Here we provide  an alternative proof by applying the open mapping theorem.
\begin{lemma}\label{maximum modulus principle for real part}
Let $f:B(0,R)\rightarrow\mathbb H$ be a regular function. If ${\rm{Re}}f$ attains its maximum at some point $q_0$, then $f$ is constant in $B(0,R)$.
\end{lemma}
\begin{proof}
 We argue by contradiction, and suppose that $f$ is not constant while ${\rm{Re}}f$ attains its maximum at some point $q_0=x_0+y_0I_0$. The real part  ${\rm{Re}}(f(x+yI_0))$ of the restriction $f_{I_0}$ of $f$ to $B_{I_0}(0,R)$ is a harmonic function of two variables $x$, $y$. Thus ${\rm{Re}}f_{I_0}$ is constant on $B_{I_0}(0,R)$. In particular, ${\rm{Re}}(f_{I_0})$ attains its maximum at the point $x_0\in (-R, R)$, which belongs to $B_I(0,R)$ for all $I\in \mathbb S$. As a consequence, ${\rm{Re}}f$ is constant on $B_I(0,R)$ for all $I\in \mathbb S$ and hence on $B(0,R)$. On the contrary, from our assumption that $f$ is not constant it follows that $f$ is not constant either, in view of the open mapping theorem. The contradiction concludes the proof.
\end{proof}
\begin{rem}
The preceding lemma can also be  proved  by using the splitting lemma for slice regular functions and the maximum modulus principle for harmonic functions.
\end{rem}

\begin{proof}[Proof of Theorem $\ref{CT4}$]
The proof of inequality (\ref{115}) is similar to the one given in \cite{CGS1}, but simpler.
The result is obvious if $f$ is constant. Otherwise, its real part ${\rm{Re}}f$ is not constant either, in view of the open mapping theorem, then it follows from Lemma \ref{maximum modulus principle for real part} that
$${\rm{Re}}f(q)<A,\qquad \forall \,\,q\in \mathbb B.$$
First, we assume that $f(0)=0$. Consider the function
$$g(q)=\big(2A-f(q)\big)^{-\ast}\ast f(q)=f(q)\ast\big(2A-f(q)\big)^{-\ast},$$
which is regular on $\mathbb B$ and $|g(q)|<1$. The last assertion follows from the facts that $g=\big(2A-f\circ T_{2A-f}\big)^{-1}f\circ T_{2A-f}$ and ${\rm{Re}}f<A$. Since $g(0)=0$, it follows from Theorem \ref{th:SL-theorem} that
\begin{eqnarray}\label{eq:31}
|g(q)|\leq|q|, \qquad \forall\ q\in\mathbb B,
\end{eqnarray}
and
\begin{eqnarray}\label{eq:32}
|g'(0)|\leq1.
\end{eqnarray}

From the very definition of $g$ and Proposition \ref{prop:Quotient Relation} it follows that
$$f(q)=2Ag(q)\ast(g(q)-1)^{-\ast}=2A(g(q)-1)^{-\ast}\ast g(q)=2A\big(g\circ T_{g-1}(q)-1\big)^{-1}g\circ T_{g-1}(q),$$
and hence
\begin{equation}\label{eq:33}
|f(q)|\leq2A\frac{|T_{g-1}(q)|}{1-|T_{g-1}(q)|}=2A\frac{|q|}{1-|q|}, \qquad \forall\ q\in\mathbb B.
\end{equation}
Moreover,
\begin{equation}\label{156}
\begin{split}
   {\rm{Re}}f(q)&=2A\frac{|g\circ T_{g-1}(q)|^2-{\rm{Re}}\big(g\circ T_{g-1}(q)\big)}{|1-g\circ T_{g-1}(q)|^2} \\
   &=A\bigg(1-\frac{1-|g\circ T_{g-1}(q)|^2}{1+|g\circ T_{g-1}(q)|^2-2{\rm{Re}}\big(g\circ T_{g-1}(q)\big)}\bigg)\\
   &\leq A\bigg(1-\frac{1-|g\circ T_{g-1}(q)|^2}{(1+|g\circ T_{g-1}(q)|)^2}\bigg)\\
   &=2A\frac{|g\circ T_{g-1}(q)|}{1+|g\circ T_{g-1}(q)|}\\
   &=2A\bigg(1-\frac{1}{1+|g\circ T_{g-1}(q)|}\bigg)\\
   &=2A\bigg(1-\frac{1}{1+|q|}\bigg)\\
   &=2A\frac{|q|}{1+|q|},
\end{split}
\end{equation}
since $|g\circ T_{g-1}(q)|\leq|T_{g-1}(q)|=|q|$, which follows from $(\ref{eq:31})$.

Again from the very definition of $g$, we obtain that
 $$(2A-f)\ast g=f,$$
 which together with Proposition \ref{th:Leibniz-rule} implies that
$$f'\ast(1+g)=(2A-f)\ast g'.$$
We now  evaluate  the preceding identity at $q=0$. Since $f(0)=0$ and $g(0)=0$, applying Proposition \ref{prop:RP} we obtain
\begin{eqnarray}\label{eq:34}
 f'(0)=2Ag'(0),
 \end{eqnarray}
 which together with (\ref{eq:32}) yields
\begin{eqnarray}\label{eq:35}
|f'(0)|\leq2A.
 \end{eqnarray}

For general case, we consider the function $f-f(0)$, replacing $f$, $A$ by $f-f(0)$ and $A-{\rm{Re}}f(0)$ in inequalities (\ref{eq:33}), (\ref{156}) and(\ref{eq:35}) respectively, yields that
\begin{equation}\label{eq:36}
|f(q)-f(0)|\leq2(A-{\rm{Re}}f(0))\frac{|q|}{1-|q|}, \qquad \forall\ q\in\mathbb B,
\end{equation}
\begin{equation}\label{157}
 {\rm{Re}}f(q)\leq {\rm{Re}}f(0)+\frac{2|q|}{1+|q|}\big(A-{\rm{Re}}f(0)\big)
 \leq \frac{2|q|}{1+|q|}A+\frac{1-|q|}{1+|q|}{\rm{Re}}f(0), \quad \forall\ q\in\mathbb B,
\end{equation}
and
\begin{eqnarray}\label{eq:37}
|a_1|=|f'(0)|=|\big(f-f(0)\big)'(0)|\leq2\big(A-{\rm{Re}}f(0)\big).
 \end{eqnarray}
Now inequality (\ref{115}) follows from the maximum modulus principle and (\ref{eq:36}). Similarly, inequality (\ref{155}) follows from Lemma \ref{maximum modulus principle for real part} and (\ref{157}).

Next, we want to prove that $|a_{n_0}|\leq 2(A-{\rm{Re}}f(0))$ for any fixed  $n_0\geq 2$. The proof is similar to the one given in Theorem \ref{CT3}.
To this end, we need to
construct a regular function $\varphi$ with the same properties as $f$, whose Taylor coefficient of the first degree term is $a_{n_0}$ so that   we would conclude  that $|a_{n_0}|\leq 2(A-{\rm{Re}}f(0))$.

Starting from the Taylor expansion of  $f$, given by  $f(q)=\sum\limits_{n=0}^\infty q^n a_n$,
we set
\begin{equation}\label{eq:38}
\varphi(q)=\sum\limits_{m=0}^\infty q^{m} a_{mn_0}=a_0+qa_{n_0}+q^2a_{2n_0}+\cdots,
\end{equation}
and
\begin{equation}\label{eq:39}
h(q)=\varphi(q^{n_0}).
\end{equation}
One can see from the radii of convergence of Taylor expansions that $\varphi$ and $h$ are slice regular functions on $\mathbb B$, respectively.

Reasoning as in the proof of Theorem \ref{CT3} gives that the restriction $h_I$ of $h$ to  $\mathbb B_I$ with $I\in\mathbb S$ is exactly given by
\begin{equation}\label{eq:40}
h_I(z_I)=\frac 1 {n_0} \Big(f_I(z_I)+f_I(z_I\omega_I)+\cdots +f_I(z_I\omega_I^{n_0-1})\Big),\quad \forall \,\,z_I\in \mathbb B_I,
\end{equation}
where $f_I$ is the restriction of $f$ to $\mathbb B_I$
and  $\omega_I\in \mathbb C_I$ is any  quaternionic primitive $n_0$-th root of unity.
It follows from (\ref{eq:40}) that
$$A_h:=\sup\limits_{q\in\mathbb B}{\rm{Re}}(h(q))\leq A_f=A,$$
and hence
$$A_\varphi:=\sup\limits_{q\in\mathbb B}{\rm{Re}}(\varphi(q))\leq A_f=A$$ by the relation (\ref{eq:39}).
Moreover, we have $\varphi(0)=a_0=f(0)$.

As a result, the regular function $\varphi$  constructed via $f$ has the same properties as $f$, whose Taylor coefficient of the first degree term is $a_{n_0}$. Consequently, we obtain that $|a_{n_0}|=|\varphi'(0)|\leq 2(A_\varphi-{\rm{Re}}\varphi(0))\leq 2(A-{\rm{Re}}f(0))$ for any $n_0\in\mathbb N$.

Finally, we come to prove inequality (\ref{116}). It follows from the Taylor expansion of $f$ that $$f^{(n)}(q)=\sum\limits_{m=n}^{\infty}m(m-1)\cdots (m-n+1)q^{m-n} a_m,$$
and hence by (\ref{114}),
\begin{equation}\label{41}
\begin{split}
  |f^{(n)}(q)|&\leq\sum\limits_{m=n}^{\infty}m(m-1)\cdots (m-n+1)|a_m||q|^{m-n} \\
    &\leq 2(A-{\rm{Re}}f(0))\sum\limits_{m=n}^{\infty}m(m-1)\cdots (m-n+1)r^{m-n}\\
    &=2(A-{\rm{Re}}f(0))\left.\bigg(\frac{d^n}{d t^n}\sum\limits_{m=0}^{\infty}t^m\bigg)\right|_{t=r}\\
    &=\frac{2n!}{(1-r)^{n+1}}\big(A-{\rm{Re}}f(0)\big)
\end{split}
\end{equation}
for all $|q|\leq r<1$ and $n\in\mathbb N$. Incidentally, inequality (\ref{115}) can also proved in the same manner as above. Now the proof is completed.
\end{proof}

Finally we come to show the equivalence of Theorems \ref{CT3} and  \ref{CT4}.
\begin{thm}
Theorems $\ref{CT3}$ and $\ref{CT4}$ are equivalent.
\end{thm}
\begin{proof}
Set $A=0$ and consider the function $-f$, then $(\ref{eq:11})$ and $(\ref{eq:11an})$ follow from Theorem \ref{CT4}.

Now we deduce Theorem \ref{CT4} from Theorem \ref{CT3}. Let $f$ be described in Theorem \ref{CT4} and be not constant, then
$${\rm{Re}}f(q)<A,\qquad \forall \,\,q\in \mathbb B.$$
Consider the function $$g(q)=\frac{\big(A+{\rm{Re}}f(0)\big)-\big(f(q)+\overline{f(0)}\big)}{A-{\rm{Re}}f(0)},\qquad \forall \,\,q\in \mathbb B.$$
Then $g$ is regular on $\mathbb B$, $g(0)=1$ and $${\rm{Re}}g(q)=\frac{A-{\rm{Re}}f(q)}{A-{\rm{Re}}f(0)}>0,$$
i.e. $g$ satisfies the assumptions given in Theorem \ref{CT3}.
Therefore,
$$|a_n|=\bigg|\frac{f^{(n)}(0)}{n!}\bigg|
=\bigg|\frac{g^{(n)}(0)}{n!}\bigg|\big(A-{\rm{Re}}f(0)\big)\leq2\big(A-{\rm{Re}}f(0)\big),$$
and
$$\frac{A-{\rm{Re}}f(q)}{A-{\rm{Re}}f(0)}={\rm{Re}}g(q)>\frac{1-|q|}{1+|q|},$$
from which it follows that
$${\rm{Re}}f(q)\leq \frac{2|q|}{1+|q|}A+\frac{1-|q|}{1+|q|}{\rm{Re}}f(0), \quad \forall\ q\in\mathbb B.$$
Now inequality (\ref{155}) follows from Lemma \ref{maximum modulus principle for real part}. Inequalities (\ref{115}) and (\ref{116}) follow from  (\ref{114}) as in the proof of Theorem \ref{CT4}. This completes the proof.
\end{proof}

\bibliographystyle{amsplain}

\end{document}